\documentclass[letterpaper, 10pt, conference]{ieeeconf}      

\IEEEoverridecommandlockouts                              
\usepackage{amsthm}
\usepackage{xcolor}
\usepackage{amsmath}
\usepackage{amssymb}
\usepackage{comment}
\usepackage{amsfonts}
\usepackage{graphicx}
\usepackage{subcaption}
\usepackage{romannum}
\let\labelindent\relax
\usepackage{enumitem}
\usepackage{algorithm}
\usepackage{algpseudocode}

\newcommand{\R}{\mathbb{R}}
\newcommand{\mc}[1]{\mathcal{#1}}
\newcommand{\mb}[1]{\mathbf{#1}}
\newcommand{\mbb}[1]{\mathbb{#1}}
\newcommand{\bx}{\mathbf{x}}
\newtheorem{cor}{Corollary}
\newcommand{\RN}[1]{\mathrm{\uppercase\expandafter{\romannumeral#1}}}
\newtheorem{assump}{Assumption}
\newtheorem{thm}{Theorem}
\newtheorem{remark}{Remark}
\newenvironment{pf}{\begin{proof}}{\end{proof}}

\title{\LARGE \bf
Bundle EXTRA for Decentralized Optimization
}

\author{Haijuan Liu, Zhuoqing Zheng, Cong Li, Wenying Xu, and Xuyang Wu
\thanks{H. Liu, Z. Zheng, C. Li and X. Wu are with the School of Automation and Intelligent Manufacturing, Southern University of Science and Technology, Shenzhen, China, and the State Key Laboratory of Autonomous Intelligent Unmanned Systems, Beijing 100081, China. Email: {\tt\small 12431368@mail.sustech.edu.cn; 12433033@mail.sustech.edu.cn; licong@sustech} {\tt\small.edu.cn; wuxy6@sustech.edu.cn}}
\thanks{W. Xu is with the School of Mathematics, Southeast University, Nanjing 211189, China. Email: {\tt\small wyxu@seu.edu.cn}}
\thanks{This work is supported in part by the Guangdong Provincial Key Laboratory of Fully Actuated System Control Theory and Technology under grant No. 2024B1212010002, in part by the Shenzhen Science and Technology Program under grant No. JCYJ20241202125309014, in part by the Shenzhen Science and Technology Program 
under grant No. KQTD20221101093557010, and in part by the Guangdong Basic and Applied Basic Research Foundation under Grant No. 2026A1515012017.}
}

\begin{document}

\maketitle
\thispagestyle{empty}
\pagestyle{empty}

\begin{abstract}
Decentralized primal-dual methods are widely used for solving decentralized optimization problems, but their updates often rely on the potentially crude first-order Taylor approximations of the objective functions, which can limit convergence speed. To overcome this, we replace the first-order Taylor approximation in the primal update of EXTRA, which can be interpreted as a primal-dual method, with a more accurate multi-cut bundle model, resulting in a fully decentralized bundle EXTRA method. The bundle model incorporates historical information to improve the approximation accuracy, potentially leading to faster convergence. Under mild assumptions, we show that a KKT residual converges to zero. Numerical experiments on decentralized least-squares problems demonstrate that, compared to EXTRA, the bundle EXTRA method converges faster and is more robust to step-size choices.
\end{abstract}

\section{INTRODUCTION}
Decentralized optimization employs a network of agents to solve a global optimization problem, where each agent can communicate only with its neighbors. This problem has attracted much attention in the last decades due to its broad applications in diverse research areas such as distributed control \cite{olfati2007con} and distributed machine learning \cite{lian2017can}.

In the category of decentralized optimization methods, early methods are mainly primal and based on the consensus operation, such as the decentralized gradient descent method (DGD) \cite{nedic2009distributed,yuan2016convergence,Jakovetic2014fast}. Although these methods can be well adapted to a wide range of scenarios, such as time-varying network \cite{nedic2009distributed}, stochastic optimization \cite{lian2017can}, and asynchronous optimization \cite{assran2020asynchronous,wu2026asynchronous}, their convergence properties are unsatisfactory. In particular, when using fixed step-sizes, these methods are only guaranteed to converge to a sub-optimal solution. Although theoretically these methods converge to the optimum when using diminishing step-sizes, the quickly vanishing step-size often yields slow convergence.

To achieve faster convergence, a list of advanced methods \cite{shi2015extra,nedic2017achieving,shi2015proximal,Nedic2017geometrically,Qu2018Harnessing,xi2017DEXTRA} are proposed, such as EXTRA \cite{shi2015extra}, DIGing \cite{Nedic2017geometrically}, and several primal-dual methods \cite{alghunaim2020linear,Mak2017con,lei2016primal,aybat2017distributed}, where EXTRA and DIGing are also shown to have a primal-dual interpretation \cite{wu2022unifying}. These methods can converge to the exact optimum with constant step-sizes and enjoy faster convergence both theoretically and empirically. However, most of their updates involve approximating the primal/dual function by its first-order Taylor expansion, which can yield a crude approximation and prevent faster convergence. With this observation, a natural idea for accelerating these methods is to use a more accurate surrogate. 

In this paper, we view EXTRA as a primal-dual method, and replace the first-order Taylor expansion of the objective function with a surrogate function referred to as the \emph{bundle model} in its primal update, leading to a decentralized bundle EXTRA method. Our method allows for several choices of the bundle model, which incorporate lower bounds of the objective function or historical function values and gradients to improve the accuracy of the surrogate. The main contributions of this paper are summarized as follows:
\begin{enumerate}
    \item We adapt the bundle technique in centralized optimization to decentralized optimization to accelerate the convergence of EXTRA. Although an earlier work \cite{wang2017decentralized} also applies the bundle technique in decentralized optimization, it focuses on non-smooth problems, allows only one specific bundle model, and does not provide any theoretical convergence guarantees.
    \item Under mild assumptions, we prove an $O(1/k)$ convergence rate of the bundle EXTRA method.
    \item Numerical experiments on decentralized least squares demonstrate that the proposed method not only converges faster but also is more robust in the step-sizes compared to  EXTRA \cite{shi2015extra}.
\end{enumerate}

The remainder of this paper is organized as follows: Section \ref{sec:Con_EXTRA} introduces consensus optimization and the EXTRA algorithm. Section \ref{sec:B_EXTRA} develops the bundle EXTRA algorithm and section \ref{sec:Con_analysis} analyses its convergence. Section \ref{sec:Num_ex} presents numerical experiments and section \ref{sec:Conclu} concludes the paper.

\textbf{Notations and definitions.}
We use $\mathbb{R}^d$ and $\mathbb{R}^{n\times d}$ to denote the $d$-dimensional Euclidean space and the set of $n\times d$ real matrices, respectively. Additionally, $\langle \cdot,\cdot \rangle$ and $\|\cdot\|$ represent the inner product and the Euclidean norm, respectively.
For any matrix $A\in\R^{n\times d}$, $\operatorname{Null}(A)=\{x\in\R^d|Ax=0\}$ and $\operatorname{Range}(A) = \{ Ax \mid x \in \mathbb{R}^d \}$ are the null space and range of $A$, respectively, and $A^\dagger$ denotes its Moore--Penrose pseudoinverse. For two matrices $A, B\in \R^{n\times n}$, $A\succ B$ ($A\succeq B$) means that $A-B$ is positive definite (positive semidefinite). We use  $\mb{1}_d$, $\mb{0}_{n\times d}$, and $I_n$ to denote the $d$-dimensional all-one vector, the $n\times d$ zero matrix and the $n$ dimensional identity matrix, respectively, and ignore the subscripts when they are clear from the context. For any $x\in\mathbb{R}^n$ and symmetric and positive semidefinite matrix $A\in\R^{n\times n}$, the weighted norm $\|x\|_A= \sqrt{x^T A x}$. For any vector $v\in\R^d$, we use $\operatorname{span}(v)$ to denote its span. For any $f: \R^d \to \R$, $\partial{f}(\cdot)$ represents the subdifferential of $f$, and we say it is $L$-smooth for some $L>0$ if it is differentiable and
\[\|\nabla f(y)-\nabla f(x)\|\le L\|y-x\|, \forall~x, y\in\R^d.\]

\section{Consensus Optimization and EXTRA}\label{sec:Con_EXTRA}

This section first formulates the consensus optimization problem, and then reviews EXTRA that is a foundation of the algorithm development in Section \ref{sec:B_EXTRA}.

\subsection{Consensus optimization}

Consider a network $\mc{G}=(\mc{V}, \mc{E})$ of agents, where $\mc{V}=\{1,\ldots,n\}$ is the vertex set and $\mc{E}$ is the edge set. Consensus optimization solves the following problem through the collaboration of all agents:
\begin{equation}\label{eq:prob}
\begin{aligned}
&\underset{x_i\in\R^d}{\operatorname{minimize}}~\sum_{i=1}^n f_i(x_i), \\
&\operatorname{subject~to}~x_1=x_2=\cdots=x_n,
\end{aligned}
\end{equation}
We aim to solve \eqref{eq:prob} over a network $\mc{G}=(\mc{V}, \mc{E})$ of nodes, where $\mc{V}=\{1,\ldots,n\}$ is the vertex set and $\mc{E}$ is the edge set. To this end, we make the following assumptions.

\begin{assump}\label{assu:convex}
    The objective function $f_i$ is proper, closed, convex, and $L$-smooth for some $L>0$.
\end{assump}
\begin{assump}\label{assu:sol_exists}
    There exists at least one optimal solution to problem \eqref{eq:prob}.
\end{assump}
\begin{assump}\label{assu:connected}
    The network $\mc{G}$ is undirected and connected.
\end{assump}

Assumptions \ref{assu:convex}-\ref{assu:connected} are standard in decentralized optimization and are required in the convergence analysis of many typical decentralized optimization methods, such as DGD \cite{yuan2016convergence}, EXTRA \cite{shi2015extra}, and DIGing \cite{nedic2017achieving}. 

\subsection{EXTRA}\label{ssec:extra}

EXTRA is a typical decentralized method for solving problem \eqref{eq:prob}. To introduce it, we define
\begin{equation}
\begin{split}
    &\bx = \begin{pmatrix}
        x_1^T\\
        \vdots\\
        x_n^T
    \end{pmatrix}\in\mathbb{R}^{n\times d},\quad f(\bx)=\sum_{i=1}^n f_i(x_i),\\
    &\nabla f(\bx) = \begin{pmatrix}
        (\nabla f_1(x_1))^T\\
        \vdots\\
        (\nabla f_n(x_n))^T
    \end{pmatrix}\in\mathbb{R}^{n\times d}.
\end{split}
\end{equation}
Then, EXTRA can be described as:
\begin{equation}\label{eq:EXTRA_init}
    \mathbf{x}^1=W\mathbf{x}^0 - \alpha \nabla f(\mathbf{x}^0).
\end{equation}
For each $k\ge 0$,
\begin{equation}\label{eq:original_extra}
\begin{aligned}
    \mathbf{x}^{k+2}=&(I+W)\mathbf{x}^{k+1}-\tilde{W}\mathbf{x}^k\\
    &-\alpha(\nabla f(\mathbf{x}^{k+1})-\nabla f(\mathbf{x}^k)),
\end{aligned}
\end{equation}
where $W,\tilde{W}\in \mathbb{R}^{n\times n}$ are two weight matrices satisfying the following assumption.
\begin{assump}\label{asm:weight_matrices}
The matrices \( W = [w_{ij}]_{n \times n} \) and \( \tilde{W} = [\tilde{w}_{ij}]_{n \times n} \) satisfy
    \begin{enumerate}[label=(\alph*)]
        \item (Decentralized) If \( i \neq j \) and \((i, j) \notin \mathcal{E} \), then \( w_{ij} = \tilde{w}_{ij} = 0 \).
        \item (Symmetry) \( W = W^T, \, \tilde{W} = \tilde{W}^T \).
        \item (Null space) \(\operatorname{Null}(W-\tilde{W})= \operatorname{span}(\mathbf{1})\), \(\operatorname{Null}(I-\tilde{W})\supseteq \operatorname{span}(\mathbf{1})\).
        \item (Spectral) \( \tilde{W} \succ \mathbf{0}\) and \( \frac{I + \tilde{W}}{2} \succeq\tilde{W} \succeq W\).
    \end{enumerate}
\end{assump}
For simplicity, throughout this paper we set $\tilde{W}=\frac{W+I}{2}$ so that the EXTRA update \eqref{eq:original_extra} becomes
\begin{equation}\label{eq:Wtilde_W_extra}
\begin{aligned}
    \mathbf{x}^{k+2}=2\tilde{W}\mathbf{x}^{k+1}\!-\!\tilde{W}\mathbf{x}^k-\alpha\big(\nabla f(\mathbf{x}^{k+1})-\nabla f(\mathbf{x}^k)\big).
\end{aligned}
\end{equation}
With the setting $\tilde{W}=(W+I)/2$, Assumption \ref{asm:weight_matrices} can be guaranteed if Assumption \ref{assu:connected} holds and $W$ satisfies 1) Assumption \ref{asm:weight_matrices}(a)--(b); 2) $w_{ij}>0$ for all $i=j$ and $\{i,j\}\in\mc{E}$; 3) $W\mb{1}=\mb{1}$. Typical choices of $W$ include
\begin{enumerate}   
    \item \textbf{Metropolis weights.}
    Let $d_i$ be the degree of agent $i$. The $ij$-element of $W$ is
    \begin{equation*}
    w_{ij}=
    \begin{cases}
    \frac{1}{1+\max\{d_i,d_j\}}, & \{i,j\}\in\mathcal E,\ i\neq j,\\
    0, & \{i,j\}\notin\mathcal E,\ i\neq j,\\
    1-\sum_{j\in\mathcal N_i} w_{ij}, & i=j.
    \end{cases}
    \end{equation*}
    \item\textbf{Laplacian-based weights \cite{XIAO200465}.}
    Let $L_{\mc{G}}$ be the Laplacian matrix of $\mc{G}$ where the $ij$-element is
    \begin{equation*}
    [L_{\mc{G}}]_{ij}=
    \begin{cases}
    -1 & \{i,j\}\in\mathcal E,\ i\neq j,\\
    0, & \{i,j\}\notin\mathcal E,\ i\neq j,\\
    d_i, & i=j.
    \end{cases}
    \end{equation*}
    The matrix is
    \[
    W = I - \tau L_{\mc{G}},
    \]
    where $0 < \tau < 2/\lambda_{\max}(L_{\mc{G}})$ with $\lambda_{\max}(L_{\mc{G}})$ being the largest eigenvalue of $L_{\mc{G}}$.
\end{enumerate}
More options of $W$ are discussed in 
\cite[Section 2.4]{shi2015extra}.

It is shown that the EXTRA update \eqref{eq:Wtilde_W_extra} with the initialization \eqref{eq:EXTRA_init} can be rewritten as a primal-dual update \cite{wu2022unifying}:
\begin{align}
    \bx^{k+1} =& \arg\min_{\bx}~ f(\bx^k)+\langle \nabla f(\bx^k), \bx-\bx^k\rangle + \langle \bx, \mb{q}^k\rangle \notag\\
    &+\frac{1}{2\alpha} \|\bx-\tilde{W}\bx^k\|^2,\label{eq:x_EXTRA}\\
    \mathbf{q}^{k+1} &= \mathbf{q}^k + \frac{1}{\alpha} (I-\tilde{W})\mathbf{x}^{k+1},\label{eq:u_EXTRA}
\end{align}
where $\mb{q}^k$ is the dual iterate with
\begin{equation}
    \mb{q}^0 = \frac{1}{\alpha}(I-\tilde{W})\mb{x}^0.
\end{equation}

\section{Algorithm Development}\label{sec:B_EXTRA}

\begin{figure*}[t]
\vspace{-0.1cm}
    \centering
    \begin{subfigure}[t]{0.3\textwidth}
        \centering
        \includegraphics[width=\linewidth,height=30mm]{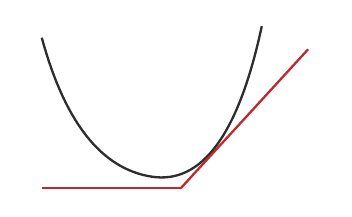}
        \caption{Polyak model.}
        \label{fig:1a}
    \end{subfigure}\hfill
    \begin{subfigure}[t]{0.3\textwidth}
        \centering
        \includegraphics[width=\linewidth,height=30mm]{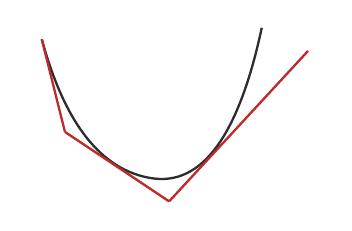}
        \caption{Cutting-plane model.}
        \label{fig:1b}
    \end{subfigure}\hfill
    \begin{subfigure}[t]{0.3\textwidth}
        \centering
        \includegraphics[width=\linewidth,height=30mm]{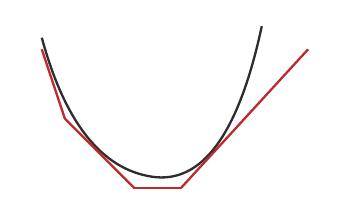}
        \caption{Polyak cutting-plane model.}
        \label{fig:1c}
    \end{subfigure}
    \caption{Surrogate function $\tilde{f}^k_i$.}
    \label{fig:model}
\end{figure*}

In the primal update \eqref{eq:x_EXTRA} of EXTRA, the function $f$ is approximated by its first order Taylor expansion

which, in general, only admits a low approximation accuracy and limits faster convergence of the algorithm. To accelerate EXTRA, we replace the first-order Taylor expansion with a more accurate approximation $\tilde{f}^k$, leading to
\begin{align}
    &\bx^{k+1}=\arg\min_{\bx}~ \tilde{f}^k(\bx) + \langle \bx, \mb{q}^k\rangle +\frac{1}{2\alpha} \|\bx-\tilde{W}\bx^k\|^2,\label{eq:x_bundle}\\
    & \mb{q}^{k+1}=\mb{q}^k+\frac{1}{\alpha}(I-\tilde{W}) \bx^{k+1}.\label{eq:q_bundle}
\end{align}
We require $\tilde{f}^k$ to have a separable structure for distributed implementation: for a set of $\tilde{f}_i^k$,
\[\tilde{f}^k(\bx)=\sum_{i=1}^n \tilde{f}_i^k(x_i).\]
Then, the algorithm can be implemented as
\begin{alignat}{2}
    x_i^{k+1}
    &= \operatorname*{arg\,min}_{x_i}\;\Big\{\tilde{f}_i^k(x_i)+ (q_i^k)^T x_i
    &\nonumber\\
    &\qquad +\frac{1}{2\alpha} \Big\|x_i-\sum_{j\in \mc{N}_i\cup\{i\}} \tilde{w}_{ij}x_j^k\Big\|^2\Big\},
    &\label{eq:x_bundle_extra}\\
    q_i^{k+1}
    &= q_i^k+\frac{1}{\alpha}\left(x_i^{k+1}-\sum_{j\in\mc{N}_i\cup\{i\}} \tilde{w}_{ij} x_j^{k+1}\right),
    &\label{eq:q_bundle_extra}
\end{alignat}
where $\mc{N}_i=\{j\mid \{i,j\}\in\mc{E}\}$ is the neighbor set of agent $i$.

We require each $\tilde{f}_i^k$ to satisfy the following assumption.
\begin{assump}\label{assu:bundle}
The surrogate function \(\tilde{f}_i^k(x_i)\) satisfies
\begin{enumerate}[label=(\alph*)]
    \item \(\tilde{f}_i^k(x_i)\) is convex;
    \item \(\tilde{f}_i^k(x_i)\ge f_i(x_i^k)+\langle \nabla f_i(x_i^k), x_i-x_i^k\rangle\) for all \(x_i\in\mbb{R}^d\);
    \item \(\tilde{f}_i^k(x_i)\le f_i(x_i)\) for all \(x_i\in \mbb{R}^d\).
\end{enumerate}
\end{assump}
Assumption \ref{assu:bundle} requires each local surrogate function to be a convex minorant of $f_i$ that dominates the first-order cutting plane at $x_i^k$. Moreover, Assumption \ref{assu:bundle} (b) and (c) imply $\tilde{f}_i^k(x_i^k)=f_i(x_i^k)$, i.e., the model is exact at $x_i^k$. Since Assumption \ref{assu:bundle} is usually assumed by bundle methods \cite{Mateo2023opti,cederberg2025asyn,iutzeler2020asynchronous}, we refer to surrogate functions satisfying Assumption \ref{assu:bundle} as a \emph{bundle model} and the algorithm \eqref{eq:x_bundle}--\eqref{eq:q_bundle} with a bundle model as the bundle EXTRA method. A detailed implementation is given in Algorithm \ref{alg:bundle_extra}.
\begin{algorithm}[h]
\caption{Bundle EXTRA}
\label{alg:bundle_extra}
\begin{algorithmic}[1]
\State\textbf{Initialization}: Determine the step-size $\alpha>0$, the mixing matrix $\tilde{W}$, and the initial variables $\mb{x}^0$.
\State Each agent $i$ shares $x_i^0$ with its neighbors $j\in \mc{N}_i=\{j\mid \{i,j\}\in\mc{E}\}$ and computes $q_i^0=\frac{1}{\alpha}(x_i^0-\sum_{j\in\mc{N}_i\cup \{i\}}\tilde{w}_{ij}x_j^0)$.
\For{$k=0,1,2,\ldots$}
    \ForAll{agents $i\in\mc{V}$}
        \State Construct the surrogate function $\tilde f_i^k$.
        \State Update the primal variable $x_i^{k+1}$ by \eqref{eq:x_bundle_extra}.
        \State Share $x_i^{k+1}$ with all its neighbors $j\in\mc{N}_i$.
        \State Update $q_i^{k+1}$ by \eqref{eq:q_bundle_extra} after receiving $x_j^{k+1}$ from
        \Statex \hspace{\algorithmicindent}\quad~ all $j\in\mc{N}_i$.
    \EndFor
\EndFor
\end{algorithmic}
\end{algorithm}

\subsection{Candidates of $\tilde{f}_i^k$}\label{ssec:candidates_tilde_f}

We provide a list of candidates for $\tilde{f}_i^k$, which incorporates historical objective function values and gradients or lower bounds of the objective function to yield a high approximation accuracy on $f_i$. Under Assumption \ref{assu:convex}, all options listed below satisfy Assumption \ref{assu:bundle} and some of them are depicted in Fig. \ref{fig:model}.
\begin{enumerate}
    \item \textbf{Polyak model:} Let $\gamma_i^f$ be a lower bound of $\min_{x_i} f_i(x_i)$. The model takes the form of
    \begin{equation*}
        \tilde{f}_i^k(x_i) = \max \{ f_i(x_i^k) + \langle \nabla f_i(x_i^k), x_i-x_i^k\rangle, \gamma_i^f\}.
    \end{equation*}
    We name this function the Polyak model since steepest descent with the Polyak step-size minimizes the above surrogate function.
   \item \textbf{Cutting-plane model:} The model takes the maximum of historical cutting planes:
    \begin{equation*}
        \tilde{f}_i^k(x_i) = \max_{t\in \mc{S}_i^k} \{ f_i(x_i^t) + \langle \nabla f_i(x_i^t), x_i-x_i^t\rangle\},
    \end{equation*}
    where $\mc{S}_i^k\subseteq [0,k]$ is an index set of historical iterates. This model is used in the cutting-plane method \cite{kelley1960cutting}.
    \item \textbf{Polyak cutting-plane model:} It takes the form of
    \begin{equation*}
        \tilde{f}_i^k(x_i) = \max_{t\in \mc{S}_i^k} \{ f_i(x_i^t) + \langle \nabla f_i(x_i^t), x_i-x_i^t\rangle, \gamma_i^f\},
    \end{equation*}
    which is the maximum of the Polyak and the cutting-plane models and is referred to as the Polyak cutting-plane model.
    
    \item\textbf{Two-cut model:} $\tilde{f}_i^0(x_i)=f_i(x_i^0)+\langle \nabla f_i(x_i^0), x_i-x_i^0\rangle$. For each $k\ge 1$,
    \begin{equation*}
    \begin{split}
        \tilde{f}_i^k(x_i)=\max\{\ell_i^{k-1,k}, f_i(x_i^k)+\langle \nabla f_i(x_i^k), x_i-x_i^k\rangle\},
    \end{split}
    \end{equation*}
    where $\ell_i^{k-1,k}=\tilde{f}_i^{k-1}(x_i^k) + \langle v_i^k, x_i-x_i^k\rangle$, the subgradient $v_i^k\in \partial \tilde{f}_i^{k-1}(x_i^k)$. This model takes the maximum of the cutting planes of both $f_i$ and $\tilde{f}_i^{k-1}$ at $x_i^k$. 
\end{enumerate}
The above four models approximate $f_i$ by a piecewise linear function defined as the maximum of multiple affine functions, which is convex and yields a substantially tighter lower bound than $f(\bx^k) + \langle \nabla f(\bx^k), \bx - \bx^k \rangle$ used in EXTRA. This tighter approximation potentially yields faster convergence.

\begin{remark}
    An existing work \cite{wang2017decentralized} also proposes a decentralized bundle-type method for solving problem \eqref{eq:prob}. However, it considers the non-smooth setting (using subgradients), only allows for the two-cut model, and includes no theoretical convergence analysis. In contrast, we consider the smooth objective function, allow for a broader range of surrogate functions, and theoretically analyzed the convergence (see Section \ref{sec:Con_analysis}).
\end{remark}

\subsection{Subproblem in the primal update \eqref{eq:x_bundle_extra}}

When using the surrogate functions in Section \ref{ssec:candidates_tilde_f}, the subproblems in the primal update \eqref{eq:x_bundle_extra} can be transformed into quadratic problems and can be solved at low cost.

We first simplify the subproblem in \eqref{eq:x_bundle_extra} as
\begin{equation}\label{eq:primal_sub}
    \min _{x\in\mbb{R}^d} \max_{j=1,\ldots,m} \{a_j^Tx + b_j\} + \frac{1}{2\alpha} \|x-c\|^2,
\end{equation}
where the subscript $i$ is ignored for simplicity, $m$ is the number of affine functions in $\tilde{f}_i^k$, $a_j^Tx+b_j$ is the $j$th affine function, and $c=\sum_{j\in \mc{N}_i\cup\{i\}} \tilde{w}_{ij}x_j^k-\alpha q_i^k$. Further, the subproblem \eqref{eq:primal_sub} can be equivalently reformulated as a quadratic program:
\begin{equation}\label{eq:linear_con}
\begin{split}
     \underset{x\in\mbb{R}^d, ~\xi\in\mbb{R}}{\operatorname{minimize}}~~&\xi + \frac{1}{2\alpha} \|x-c\|^2\\
    \operatorname{subject~to}~ &\xi -a_j^Tx-b_j\ge 0, \quad j=1,\ldots,m.
\end{split}
\end{equation}
When $d$ is small, directly solving the above quadratic problem is not expensive. 

If $d$ is large, since the variable dimension of problem \eqref{eq:linear_con} is $m$, which is typically small (around $5 - 20$ in practical implementations), a more efficient way is to solve problem \eqref{eq:linear_con} by solving its low-dimensional dual problem
\begin{equation}\label{eq:sub_dual}
\begin{aligned}
    \underset{\lambda\in \mbb{R}^{m}}{\operatorname{maximize}} ~~& h(\lambda)\\
    \operatorname{subject~to} ~ & \mb{1}^T \lambda=1,\\
                              & \lambda\ge 0.
\end{aligned}
\end{equation}
where
\[
\begin{aligned}
    h(\lambda)
    &=\operatorname{inf}_{x} ~\left\{\frac{1}{2\alpha} \|x-c\|^2 + \lambda^TAx\right\}+\lambda^T b\\
    &=-\frac{\alpha\|A^T\lambda\|^2}{2}+\lambda^T(Ac+b),
\end{aligned}
\]
is the dual objective function, $A=(a_1, \ldots, a_m)^T\in \mbb{R}^{m\times d}$, and $b=(b_1,\ldots, b_m)^T\in \mbb{R}^m$. For any optimum $\lambda^\star$ of the dual problem \eqref{eq:sub_dual}, $x^\star=c-\alpha A^T\lambda^\star$ is an optimum to problem \eqref{eq:primal_sub} (see~\cite[Section 5.5]{boyd2004con}). Note that problem \eqref{eq:sub_dual} is a quadratic program with variable dimension $m$ and the constraint set is the simplex where the projection onto it can be performed at a low cost of $O(m)$ \cite{condat2016fast}. Therefore, problem \eqref{eq:sub_dual} can be solved quickly. To test the practical effectiveness of this approach, we apply FISTA to solve problem \eqref{eq:sub_dual} with randomly generated $A,b,c$ where $d=100,000$ and $m=15$, which takes only $19\sim 40$ iterations to reach a high accuracy of $10^{-7}$ ($0.1\sim 0.2$ second on a PC with the Apple M3 8-core CPU).

\section{Convergence Analysis}\label{sec:Con_analysis}
This section analyzes the convergence of the bundle EXTRA method.
\begin{thm}\label{thm:con_analy}
    Suppose that Assumptions \ref{assu:convex}-\ref{assu:bundle} hold. Let the sequences $\{\bx^k\}$ be generated by Algorithm \ref{alg:bundle_extra}. If \begin{equation}\label{eq:step}
    \alpha\le \lambda_{\min}(\tilde{W})/L,
    \end{equation}
    then the KKT residuals satisfy
    \begin{equation}\label{eq:summable1}
    \begin{aligned}
        &\sum_{k=1}^{\infty}\Big(\frac{1}{\alpha}(\bx^k)^T(I-\tilde{W})\bx^k+\frac{1}{L}\|\nabla f(\bx^k)+\mb{q}^\star\|^2\Big)\\
        &\le \frac{1}{\alpha}\|\mb{x}^0-\mb{x}^\star\|_{\tilde{W}}^2+\alpha\|\mb{q}^0+\nabla f(\bx^\star)\|_{(I-\tilde{W})^\dag}^2<\infty,
    \end{aligned}
    \end{equation}
    where $\bx^\star$ is an optimum to problem \eqref{eq:prob}.
\end{thm}
\begin{pf}
See Appendix.
\end{pf}
In \eqref{eq:summable1}, the terms $(\bx^k)^T(I-\tilde{W})\bx^k$ and $\|\nabla f(\bx^k)+\mb{q}^\star\|^2$ measure the KKT residual of problem \eqref{eq:prob}. To see this, note that problem \eqref{eq:prob} is equivalent to
\begin{equation}\label{eq:compact_prob}
\begin{split}
    \underset{\bx\in\mbb{R}^{n\times d}}{\operatorname{minimize}}~~&~f(\bx)\\
    \operatorname{subject~to}~&~(I-\tilde{W})^{\frac{1}{2}}\bx=\mb{0},
\end{split}
\end{equation}
where the constraint implies that all $x_i$'s are identical due to Assumption \ref{asm:weight_matrices}(c) and $\operatorname{Null}(I-\tilde{W})=\operatorname{Null}((I-\tilde{W})^{\frac{1}{2}})$. By the KKT conditions of problem \eqref{eq:compact_prob}, a point $\tilde{\bx}$ is optimal if and only if there exists $\tilde{\mb{q}}\in \operatorname{Range}((I-\tilde{W})^{\frac{1}{2}})$ such that
\begin{align}
    & (I-\tilde{W})^{\frac{1}{2}}\tilde{\bx}=\mb{0},\label{eq:KKT_x}\\
    & \nabla f(\tilde{\bx}) + \tilde{\mb{q}} = \mb{0}\label{eq:KKT_q}.
\end{align}
Since $\bx^\star$ is optimal to problem \eqref{eq:compact_prob}, we have
\[\nabla f(\bx^\star)\in \operatorname{Range}((I-\tilde{W})^{\frac{1}{2}}).\]
Therefore, $(\bx^k)^T(I-\tilde{W})\bx^k=0$ implies \eqref{eq:KKT_x} with $\tilde{\bx}=\bx^k$, and $\|\nabla f(\bx^k)-\nabla f(\bx^\star)\|^2=0$ indicates \eqref{eq:KKT_q} with $\tilde{\bx}=\bx^k$ and $\tilde{\mb{q}}=-\nabla f(\bx^\star)$.

Next, we show that \eqref{eq:summable1} implies the convergence of the KKT residuals $(\bx^k)^T(I-\tilde{W})\bx^k$ and $\|\nabla f(\bx^k)+\mb{q}^\star\|^2$ to $0$.

\begin{cor}\label{cor:kkt_residual_rates}
Suppose that all the conditions in Theorem \ref{thm:con_analy} hold. Let \(\{\bx^k\}\) be generated by Algorithm \ref{alg:bundle_extra}. It holds that
\begin{align*}\label{eq:conv}
 & \lim_{k\to +\infty} (\bx^k)^T(I-\tilde{W})\bx^k = 0,\\
 & \lim_{k\to +\infty} \|\nabla f(\bx^k)+\mb{q}^\star\| = 0.
\end{align*}
Moreover, 
\begin{align*}
&\min_{t\le k} (\bx^t)^T(I-\tilde{W})\bx^t  = o\left(\frac{1}{k}\right),\\
&\min_{t\le k} \|\nabla f(\bx^t)-\nabla f(\bx^\star)\|^2 = o\left(\frac{1}{k}\right),\\
&\frac{1}{k}\sum_{t=1}^k (\bx^t)^T(I-\tilde{W})\bx^t = O\left(\frac{1}{k}\right),\\
&\frac{1}{k}\sum_{t=1}^k \|\nabla f(\bx^t)-\nabla f(\bx^\star)\|^2 = O\left(\frac{1}{k}\right).
\end{align*}
\end{cor}
\begin{pf}
By \eqref{eq:summable1} in Theorem  \ref{thm:con_analy} and \cite[Proposition 3.4]{shi2015extra}, we obtain the results.
\end{pf}

\section{Numerical Experiment}\label{sec:Num_ex}
To evaluate the performance of the proposed algorithm, we consider decentralized least-squares problems of the form
\begin{equation}\label{eq:least}
\begin{split}
\min_{x \in \mathbb{R}^d}\frac{1}{2n}\sum_{i=1}^n \|P_i x-q_i\|^2,
\end{split}
\end{equation}
which is equivalent to problem \eqref{eq:prob} with $f_i(x)=\frac{1}{2n}\|P_ix-q_i\|^2$. All entries in the matrices $P_i\in\mathbb{R}^{\eta\times d}$ and vectors $q_i\in\mathbb{R}^{\eta}$ are randomly generated from Gaussian distributions: $P_i^{lj} \sim \mc{N}(2, 2)$ and $q_i^{l} \sim \mc{N}(1, 0.5)$.

The parameter settings are as follows:  We set $n=20$, $d=100$ and $\eta=6$. The communication network is randomly generated with $32$ edges and is guaranteed to be connected. We choose the mixing matrix $W$ as the Metropolis weights described in Section \ref{ssec:extra}. We set the bundle model in bundle EXTRA \eqref{eq:x_bundle_extra} as  the following cutting-plane model:
\[\tilde{f}_i^k(x_i) = \max_{\max(0, k-m)\le t\le k} f_i(x_i^t)+\langle\nabla f_i(x_i^t), x_i-x_i^t\rangle,\]
where $m$ is a non-negative integer.

We validate the effectiveness of the proposed bundle EXTRA through comparison with EXTRA. We first compare the convergence speed where the step-sizes in both methods are fine-tuned for better performance. Then, we test the robustness of methods with respect to the step-size. The results are displayed in Figs \ref{fig:con}--\ref{fig:robust}, from which we make the following observations: First, from Fig.~\ref{fig:con}, the proposed bundle EXTRA method outperforms EXTRA and the advantage becomes more evident when the number of historical cutting planes increases. One explanation for this acceleration is that the increased number of historical cutting planes enhances the approximation of the surrogate function, which further accelerates the convergence. Second, Fig.~\ref{fig:robust} shows that bundle EXTRA can converge with a much wider range of step-sizes compared with EXTRA, especially when $m$ is large, which indicates a higher robustness of the bundle EXTRA in step-size selection.

\begin{figure}[t]
    \centering
    \includegraphics[width=0.95\columnwidth]{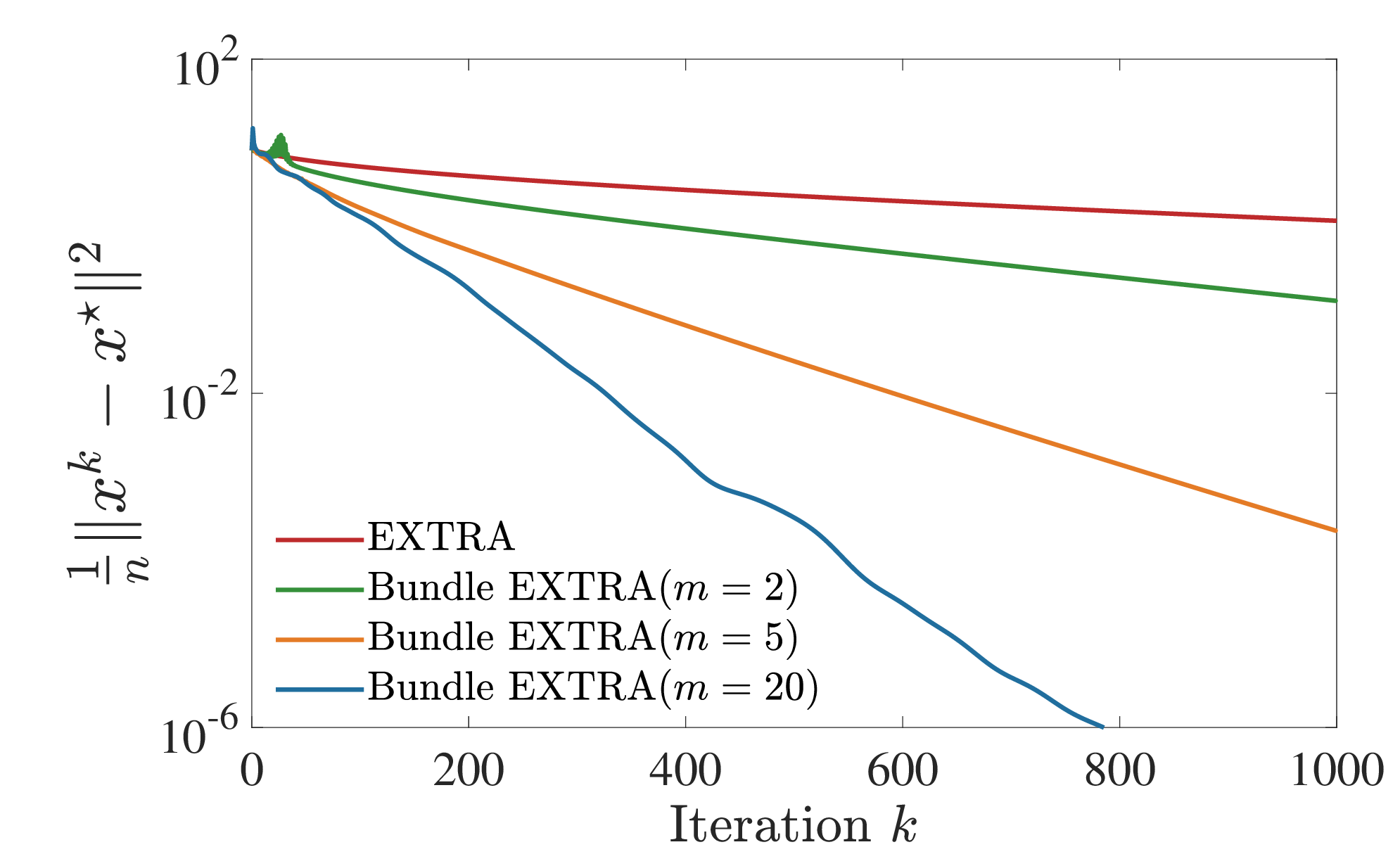}
    \caption{Convergence of bundle EXTRA and EXTRA.}
    \label{fig:con}
\end{figure}
\begin{figure}[t]
    \centering
    \includegraphics[width=0.95\columnwidth]{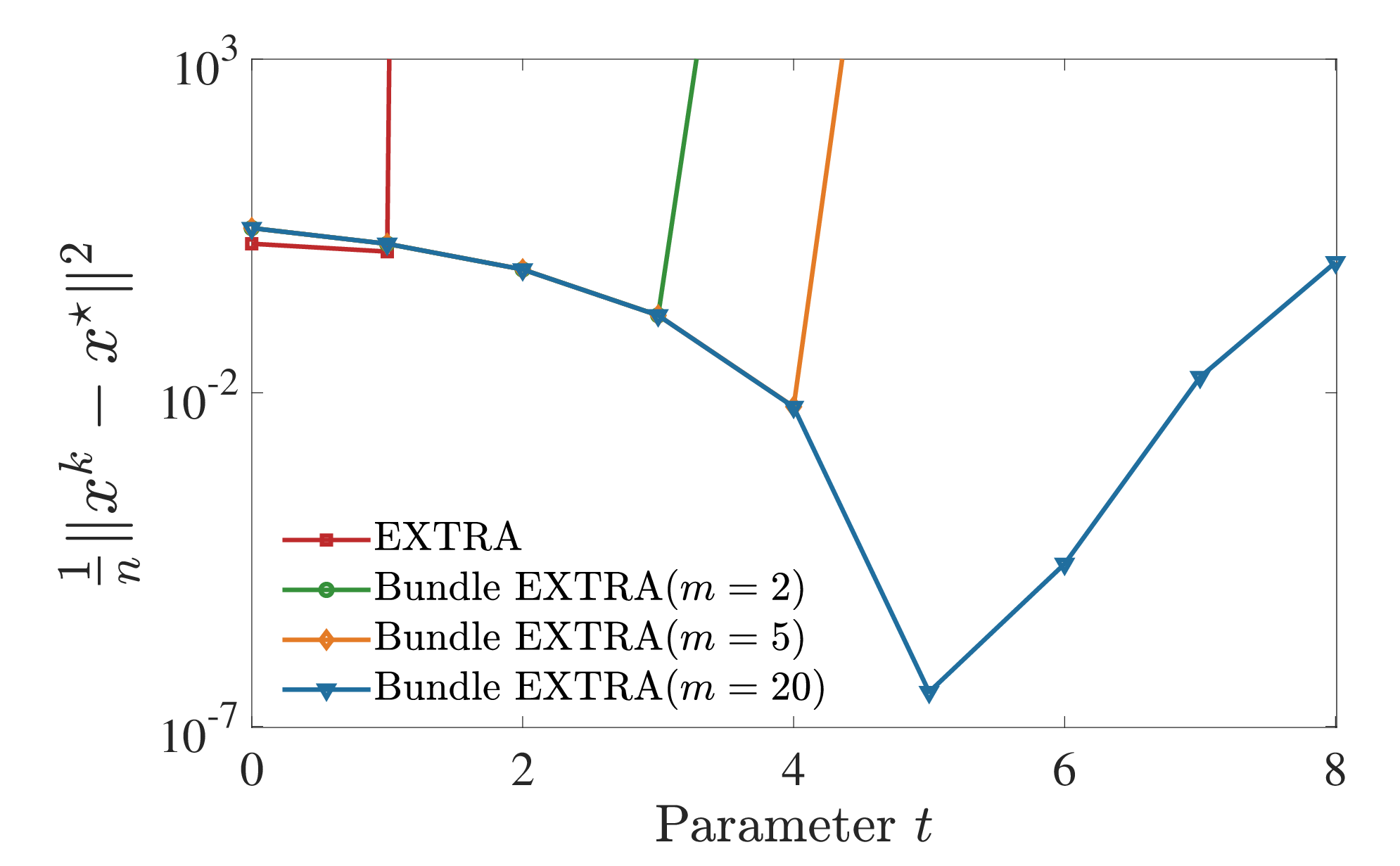}
    \caption{The parameter robustness of bundle EXTRA and EXTRA methods with step-size $\alpha= 0.003\times 2^t$.}
    \label{fig:robust}
\end{figure}
\section{Conclusion}\label{sec:Conclu}
 We have developed a bundle EXTRA method for decentralized consensus optimization. This method incorporates a bundle-based multi-cut model into the EXTRA framework, which improves the approximation accuracy of the primal objective functions while maintaining a fully decentralized structure. We have established global convergence of the KKT residual under mild assumptions. Numerical results demonstrate that our method not only converges faster, but also exhibits higher robustness in step-size selection compared to EXTRA. Future work will investigate adaptive selection of the number of cuts and more refined surrogate functions to further improve performance.

\section*{Appendix: Proof of Theorem \ref{thm:con_analy}}\label{appendix:pfthem}

We first rewrite the bundle EXTRA method in an equivalent form: Define $H=(I-\tilde{W})^{\frac{1}{2}}$. By letting $\mb{u}^k=H^\dag \mb{q}^k$, the algorithm can be equivalently rewritten as
\begin{align}
    &\bx^{k+1}=\arg\min_{\bx}~ \tilde{f}^k(\bx) + \langle \bx, H\mb{u}^k\rangle +\frac{1}{2\alpha} \|\bx-\tilde{W}\bx^k\|^2,\label{eq:x_bundle_u}\\
    & \mb{u}^{k+1}=\mb{u}^k+\frac{1}{\alpha}H\bx^{k+1},\label{eq:u_bundle}
\end{align}
where $\mb{u}^0=\frac{1}{\alpha}H\bx^0$. Since $\mb{q}^0\in \operatorname{Range}(H)$ and $\mb{q}^{k+1}-\mb{q}^k\in \operatorname{Range}(H)$ by \eqref{eq:q_bundle}, we have
\[\mb{q}^k\in \operatorname{Range}(H),\quad\forall k\ge 0,\]
so that each $\mb{q}^k$ corresponds to a unique $\mb{u}^k$. 

By the optimality conditions of problem \eqref{eq:compact_prob}, for some $\mb{u}^\star$, it holds that
\begin{equation}\label{eq:optimal}
    \left\{\begin{aligned}&\nabla f(\bx^\star)+H\mb{u}^\star=\mb{0},\\
    & H\bx^\star=\mb{0}.\end{aligned}\right. 
\end{equation}
According to the first-order optimality conditions of \eqref{eq:x_bundle_u}, for some $\mb{g}^{k+1}\in \partial \tilde{f}^k(\bx^{k+1})$, we have
    \begin{equation}\label{eq:opti_k1}
        \begin{split}
             \mb{g}^{k+1}+H\mb{u}^k+\frac{1}{\alpha} (\bx^{k+1}-\tilde{W}\bx^k)=\mb{0}.
        \end{split}
    \end{equation}
Subtracting the first equation in \eqref{eq:optimal} from \eqref{eq:opti_k1} gives
\begin{equation*}
    \mb{g}^{k+1}-\nabla f(\bx^\star) + H(\mb{u}^k-\mb{u}^\star) +\frac{1}{\alpha} (\bx^{k+1}-\tilde{W} \bx^k)=\mb{0},
\end{equation*}
or equivalently,
\begin{equation} \label{eq:Hu}
H(\mb{u}^k-\mb{u}^\star)=-(\mb{g}^{k+1}-\nabla f(\bx^\star))-\frac{1}{\alpha} (\bx^{k+1}-\tilde{W} \bx^k).
\end{equation}
By \eqref{eq:u_bundle},
\begin{equation}\label{eq:uterm}
\begin{split}
    \|\mb{u}^{k+1}-\mb{u}^\star\|^2=&\|\mb{u}^k-\mb{u}^\star + \frac{1}{\alpha}H \bx^{k+1}\|^2\\
    =&\|\mb{u}^k-\mb{u}^\star\|^2 +\frac{1}{\alpha^2} \|H\bx^{k+1}\|^2\\
    &+\frac{2}{\alpha}\langle \mb{u}^k-\mb{u}^\star, H\bx^{k+1}\rangle \\
    =&\|\mb{u}^k-\mb{u}^\star\|^2 +\|\mb{u}^{k+1}-\mb{u}^k\|^2\\
    &+\frac{2}{\alpha}\langle \mb{u}^k-\mb{u}^\star,H\bx^{k+1}\rangle.
\end{split}
\end{equation}

By \eqref{eq:Hu} and $H\bx^\star=\mb{0}$, we have
\begin{equation}\label{eq:Cross_term}
\begin{split}
    &\langle \mb{u}^k-\mb{u}^\star, H\bx^{k+1}\rangle\\ =&\langle \mb{u}^k-\mb{u}^\star,H(\bx^{k+1} -\bx^\star)\rangle\\
    =&\langle H(\mb{u}^k-\mb{u}^\star), \bx^{k+1}-\bx^\star\rangle\\
    =&-\langle \mb{g}^{k+1}-\nabla f(\bx^\star), \bx^{k+1}-\bx^\star\rangle\\
    & -\frac{1}{\alpha}\langle\bx^{k+1}-\tilde{W} \bx^k,\bx^{k+1}-\bx^\star\rangle.
\end{split}
\end{equation}
For the terms in the right-hand side of \eqref{eq:Cross_term}, we have
\begin{align}\label{eq:Wxterm}
         &-\langle \bx^{k+1}-\tilde{W} \bx^k,\bx^{k+1}-\bx^\star\rangle\notag\\
         =&-\langle \bx^{k+1}-\tilde{W} \bx^k-\tilde{W} \bx^{k+1}+\tilde{W}\bx^{k+1},\bx^{k+1}-\bx^\star\rangle\notag\\
        =&\langle \tilde{W}(\bx^{k+1}-\bx^k),\bx^\star-\bx^{k+1}\rangle-\|\bx^{k+1}-\bx^\star\|_{H^2}^2\notag\\
        =&\frac{1}{2}\big(\|\bx^k-\bx^\star\|_{\tilde{W}}^2-\|\bx^{k+1}-\bx^\star\|_{\tilde{W}}^2\big)\\
        &-\frac{1}{2}\|\bx^k-\bx^{k+1}\|_{\tilde{W}}^2-\|\bx^{k+1}-\bx^\star\|_{H^2}^2,\notag
\end{align}
 where the last step follows from the basic equality \(2\langle b-a, c-b\rangle=\|c-a\|^2-\|c-b\|^2-\|a-b\|^2\),
and by Assumption \ref{assu:bundle}, the $L$-smoothness of $f$, and \cite[Theorem 2.1.5]{nesterov2018lectures},
 \begin{equation}\label{eq:gxterm}
    \begin{split}
        &-\langle \mb{g}^{k+1}-\nabla f(\bx^\star), \bx^{k+1}-\bx^\star\rangle\\
       \le& \tilde{f}^k(\bx^\star)-\tilde{f}^k(\bx^{k+1})-\langle \nabla f(\bx^\star), \bx^\star-\bx^{k+1}\rangle\\
        \le& f(\bx^\star)-f(\bx^k)-\langle \nabla f(\bx^k), \bx^{k+1}-\bx^k\rangle\\
        &-\langle \nabla f(\bx^\star), \bx^\star-\bx^{k+1}\rangle\\
        \le& f(\bx^\star)-f(\bx^{k+1})+\frac{L\|\bx^{k+1}-\bx^k\|^2}{2}\\
        &-\langle \nabla f(\bx^\star), \bx^\star-\bx^{k+1}\rangle\\
        \le& \frac{L\|\bx^{k+1}-\bx^k\|^2}{2}-\frac{1}{2L}\|\nabla f(\bx^{k+1})-\nabla f(\bx^\star)\|^2.
    \end{split}
\end{equation}
Substituting \eqref{eq:Wxterm} and \eqref{eq:gxterm} into \eqref{eq:Cross_term} and using \[\|\bx^{k+1}-\bx^\star\|_{H^2}^2=\alpha^2\|\mb{u}^{k+1}-\mb{u}^k\|^2,\]
we have
\begin{align*}
        &\langle \mb{u}^k-\mb{u}^\star, H\bx^{k+1}\rangle\\
        \le & \frac{1}{2\alpha} \big(\|\bx^\star-\bx^k\|_{\tilde{W}}^2-\|\bx^{k+1}-\bx^\star\|_{\tilde{W}}^2-\|\bx^k-\bx^{k+1}\|_{\tilde{W}}^2\big)\\
  &-\alpha\|\mb{u}^{k+1}-\mb{u}^k\|^2-\frac{1}{2L}\|\nabla f(\bx^{k+1})-\nabla f(\bx^\star)\|^2\\
  &+\frac{L\|\bx^{k+1}-\bx^k\|^2}{2},
\end{align*}
which, together with \eqref{eq:uterm}, yields
\begin{equation}\label{eq:rho_con}
\begin{split}
        &\alpha\|\mb{u}^{k+1}-\mb{u}^\star\|^2\\
  \le &\alpha\|\mb{u}^k-\mb{u}^\star\|^2 -\alpha\|\mb{u}^{k+1}-\mb{u}^k\|^2 + L\|\bx^{k+1}-\bx^k\|^2\\
   & +\frac{1}{\alpha}\big(\|\bx^\star-\bx^k\|_{\tilde{W}}^2-\|\bx^{k+1}-\bx^\star\|_{\tilde{W}}^2\big)\\
   & -\frac{1}{\alpha}\|\bx^k-\bx^{k+1}\|_{\tilde{W}}^2
  - \frac{1}{L}\|\nabla f(\bx^{k+1})-\nabla f(\bx^\star)\|^2.
\end{split}
\end{equation}

Define 
\[ Q =
\begin{pmatrix}
\alpha I & 0 \\
0 & \frac{1}{\alpha}\tilde{W}
\end{pmatrix},
\]
and let \(\mb{z}^k=((\mb{u}^k)^T, (\bx^k)^T)^T\). We rearrange terms in \eqref{eq:rho_con} and obtain
 \begin{equation}
 \begin{split}
     &\|\mb{z}^k-\mb{z}^\star\|^2_Q-\|\mb{z}^{k+1}-\mb{z}^\star\|^2_Q\\
     =&\frac{1}{\alpha}\|\bx^k-\bx^\star\|_{\tilde{W}}^2+ \alpha \|\mb{u}^k-\mb{u}^\star\|^2 \\
     &-\frac{1}{\alpha}\|\bx^{k+1}-\bx^\star\|_{\tilde{W}}^2-\alpha\|\mb{u}^{k+1}-\mb{u}^\star\|^2 \\
     \ge& \|\bx^k-\bx^{k+1}\|^2_{\tilde{W}/\alpha-LI} +\alpha\|\mb{u}^{k+1}-\mb{u}^k\|^2 \\
     &+ \frac{1}{L}\|\nabla f(\bx^{k+1})-\nabla f(\bx^\star)\|^2.
\end{split}
 \end{equation}
Since $\alpha\le \frac{\lambda_{\min}(\tilde{W})}{L}$ as assumed in \eqref{eq:step}, it holds that
\begin{equation}\label{eq:summable}
    \begin{split}
        &\sum_{t=0}^{k} \big(\|\bx^{t+1}-\bx^t\|_{{\tilde{W}}/\alpha-LI}^2 +\alpha\|\mb{u}^{t+1}-\mb{u}^t\|^2 \\
        &+ \frac{1}{L}\|\nabla f(\bx^{t+1})-\nabla f(\bx^\star)\|^2\big) \le \|\mb{z}^0-\mb{z}^\star\|^2_Q<\infty.
    \end{split}
\end{equation}
Moreover, by \eqref{eq:optimal}, we have $\mb{u}^\star=-H^\dag\nabla f(\bx^\star)$, which, together with $\mb{u}^0=H^\dag\mb{q}^0$, yields
\[\mb{u}^0-\mb{u}^\star = H^\dag H^\dag(\mb{q}^0+\nabla f(\bx^\star)).\]
Substituting the above equation into \eqref{eq:summable} gives \eqref{eq:summable1}. This completes the proof of Theorem \ref{thm:con_analy}.


\bibliographystyle{ieeetran}
\bibliography{reference}

\end{document}